\RequirePackage{tikz-cd}
\RequirePackage{tikz}

\documentclass[12pt, english]{article}

\usepackage[utf8]{inputenc}
\usepackage{amsthm}
\usepackage{amssymb}
\usepackage[all]{xy}
\usepackage{amsmath}

\usepackage{float}

\usepackage{rotating}
\usepackage{caption, booktabs}
\usepackage{verbatim}

\usepackage{pdflscape}

\usepackage{mathtools}

\usepackage{hyperref}
\usepackage{bookmark}

\newtheorem{definition}{Definition}[section]
\newtheorem{proposition}[definition]{Proposition}
\newtheorem{remark}[definition]{Remark}
\newtheorem{lemma}[definition]{Lemma}
\newtheorem{theorem}[definition]{Theorem}
\newtheorem{corollar}[definition]{Corollary}
\newtheorem{construction}[definition]{Construction}
\newtheorem{example}[definition]{Example}

\DeclareMathOperator{\ord}{ord}

\DeclareMathOperator{\Vol}{Vol}

\begin{document}





\title{The plurigenera of nondegenerate minimal toric hypersurfaces}
\author{Julius Giesler \\ University of T\"ubingen}
\date{\today}
\maketitle
\begin{abstract}
	In this article we present a formula for the plurigenera of minimal models of nondegenerate toric hypersurfaces, which is valid in arbitrary dimension and which expresses these invariants through lattice points on the Fine interior. From this formula we derive a formula for the volume of the canonical divisor of the toric hypersurface. We also study the pluricanonical mappings and show under some restrictions how to construct a canonical model of the toric hypersurface.
\end{abstract}


\pagestyle{plain}

\section{Introduction}

In (\cite[(4.12),(4.13)]{Rei87}) it was formulated that one might use a special polyotpe, called the \textit{Fine interior}, to compute the plurigenera of nondegenerate toric hypersurfaces. This idea goes back to the doctoral thesis of J. Fine (\cite{Fine83}). But the article of M. Reid lacks an explicit formula. In this article we provide such a formula for the plurigenera of minimal models of nondegenerate toric hypersurfaces:
\begin{theorem} \label{theorem_intro_plurigenera}
	Let $\Delta$ be an $n$-dimensional lattice polytope, where $n \geq 2$, with $k:= \dim \, F(\Delta) \geq 0$ and let $Y$ be a minimal model of $Z_f$. Then for $m \geq 1$ the plurigenera $P_m(Y):= h^{0}(Y, m K_Y)$ are given by
	\[
	P_m(Y) = \left\{ \begin{array}{lll} l(m \cdot F(\Delta)) - l^*((m-1) \cdot F(\Delta)), & k = n \\
	l(m \cdot F(\Delta)) + l^*((m-1) \cdot F(\Delta)), & k = n-1 \\
	l(m \cdot F(\Delta)) & k < n-1,
	
	\end{array}\right.
	\]
\end{theorem}
Here for $F \subset M_{\mathbb{R}}$ a rational polytope $l(F) := \text{\textbar}F \cap M \text{\textbar}$ is the number of lattice point of $F$ and $l^*(F)$ the number of interior lattice points of $F$. 
\\
In the article (\cite{Bat22}) it is shown how to construct minimal models and other birational models with mild singularities of nondegenerate toric hypersurfaces. That is we start with a nondegenerate Laurent polynomial $f$ with $n$-dimensional Newton polytope $\Delta \subset M_{\mathbb{R}}$, consider the zero set $Z_f := \{f=0\} \subset (\mathbb{C}^*)^n$ and ask for a projective toric variety $\mathbb{P}$ such that the closure $Y=Y_f$ of $Z_f$ in $\mathbb{P}$ has at most terminal singularities and nef canonical divisor $K_{Y}$. In this case we call $Y$ a \textit{minimal model} of $Z_f$.
\\ \\
$\mathbb{P}$ exists if $F(\Delta) \neq \emptyset$ and in (\cite{Bat22}) it was shown how to construct $\mathbb{P}$ combinatorially. To compute the plurigenera of $Y$ in section \ref{section_plurigenera} we use an exact sequence arising from the adjunction formula 
\[ 0 \rightarrow \mathcal{O}_{\mathbb{P}}((m-1)(K_{\mathbb{P}} + Y) + K_{\mathbb{P}}) \rightarrow \mathcal{O}_{\mathbb{P}}(m(K_{\mathbb{P}} + Y)) \rightarrow \mathcal{O}_Y(mK_Y)  \rightarrow 0. \]
We exploit that the toric divisor $m(K_{\mathbb{P}}+ Y)$ is $\mathbb{Q}$-Cartier and nef, and the polytope associated to it is $m \cdot F(\Delta)$. Then we apply Serre duality on $\mathbb{P}$ and a vanishing Theorem to deduce the formula on the plurigenera.
\newline \newline
In (\cite{Fle13}) and (\cite{Rei87}) the Plurigenera are computed for $3$-folds, which are not necessarily toric hypersurfaces: If $Y$ is a projective $3$-fold with at most canonical singularities and with $K_Y$ ample, then we have for $m \geq 2$ (\cite[17.3]{Fle13})
\begin{align*}
P_m(Y) &= \chi(Y, \mathcal{O}_Y(mK_Y)) \\
&= \frac{(2m-1)m(m-1)}{12r} K_Y^3 - (2m-1) \chi(Y,\mathcal{O}_Y) + l(n).
\end{align*}
Here the numbers $r$ and $l(n)$ depend on the singularities of $Y$ and a computation of them is more involved. In (\cite{DK86}) it is shown that if $\mathbb{P}_{\Delta}$ is $\mathbb{Q}$-factorial the first plurigenus of the closure $Z_{\Delta}$ in $\mathbb{P}_{\Delta}$ is given by $l^*(\Delta)$, which equals $l(F(\Delta))$. Our formula in Theorem \ref{theorem_intro_plurigenera} is valid in aribtrary dimensions $n \geq 2$ and should be much more useful for computations. As an illustration of our formula we deal with smooth curves in toric surfaces. 
\\ \\
In Corollary \ref{corollary_volume_of_can_divisor} we derive from Theorem \ref{theorem_intro_plurigenera} a formula for the volume $K_Y^{n-1}$ of $K_Y$, which though was first and independently found by V. Batyrev. If $\dim \, F(\Delta) \in \{n-1,n\}$ then $Y$ is a variety of general type. In Corollary \ref{theorem_can_model_surf_gen_type} we show how to construct a \textit{canonical model} of $Y$ at least if $\dim \, F(\Delta) = n$. Finally we consider the pluricanonical mappings $\psi_{mK_Y}$ on $Y$ and give criteria when they are morphisms or birational, where for the latter we assume $\dim \, F(\Delta) = n$.

\section{Background on toric varieties} \label{section_background_on_tori_var}

In this article $M$ always denotes an $n$-dimensional lattice $\mathbb{Z}^n$ with dual lattice $N$. We write $M_{\mathbb{R}}$ for $M \otimes \mathbb{R}$. Let $T := N \otimes_{\mathbb{Z}} \mathbb{C}^* \cong (\mathbb{C}^*)^n$ be the $n$-dimensional torus. By $\Delta$ we denote a lattice polytope in $M_{\mathbb{R}}$ and think of it as the Newton polytope of some Laurent polynomial $f$, that is
\begin{align} \label{representation_Laurent_polynomial} 
	f = \sum\limits_{m \in M \cap \Delta} a_m x^m, \quad a_m \in \mathbb{C},
\end{align}
where $a_m \neq 0$ for $m$ a vertex of $\Delta$. Here $x^m$ is an abbreviation for $x_1^{m_1} \cdot ... \cdot x_n^{m_n}$. \\
By a rational polytpe $F \subset M_{\mathbb{R}}$ we mean a polytope, whose vertices have coordinates in $\mathbb{Q}$. We may represent a rational polytope $F$ as
\[ F = \{x \in M_{\mathbb{R}} \mid \langle x, \nu_i \rangle \geq -a_i, \, i=1,...,r  \},  \]
where $\nu_i \in N$ are primitive and $a_i \in \mathbb{Q}$. We associate to $F$ its normal fan $\Sigma_{F}$ which has as rays $\Sigma_{F}[1] = \{ \nu_1,...,\nu_r \}$. For $\Delta$ an $n$-dimensional lattice polytope or more generally a rational polytope let
\[ \ord_{\Delta}(\nu):= \min\limits_{m \in \Delta} \langle m, \nu \rangle \quad \nu \in N,  \]
such that $\ord_{\Delta}(\nu_i) = -a_i$. We denote the projective toric variety associated to the normal fan $\Sigma_F$ of $F$ by $\mathbb{P}_F$. We denote the toric divisor associated to $\nu_i$ by $D_i$.
\\ \\
For an $n$-dimensional lattice polytope $\Delta$ we denote by $L(\Delta)$ the vector space over $\mathbb{C}$ with basis the lattice points of $\Delta$ and by $L^{*}(\Delta)$ the vector space over $\mathbb{C}$ with basis the interior lattice points of $\Delta$, where for the latter $\Delta$ is considered as a subset of an affine space of the same dimension as $\Delta$. The dimensions of $L(\Delta)$ and $L^{*}(\Delta)$ are denoted by $l(\Delta)$ and $l^{*}(\Delta)$.

\begin{remark}
	\normalfont
	We write $D \sim_{lin} D'$ for the linear equivalence of Weil divisors. To a Cartier (Weil) divisor $D$ on an algebraic variety $Y$ is associated a locally free sheaf (rank $1$ reflexive sheaf) $\mathcal{O}_Y(D)$, such that $\mathcal{O}_Y(D)$ is isomorphic to $\mathcal{O}_Y(D')$ if and only if $D$ is linear equivalent to $D'$ (see \cite[App. to §1]{Rei79}).
\end{remark}

\begin{proposition}
	A projective toric variety $V$ to a fan $\Sigma$ is $\mathbb{Q}$-factorial if and only if each cone $\sigma \in \Sigma$ is simplicial, that is the generators $\nu_i \in \sigma[1]$ are linearly independent over $\mathbb{R}$.
\end{proposition}

\begin{proposition}
	Given two $n$-dimensional complete fans $\Sigma$ and $\Sigma'$ with associated toric varieties $V$ and $V'$, such that $\Sigma[1]$ and $\Sigma'[1]$ belong to the same lattice $N$ and $\Sigma'$ refines $\Sigma$ there is an induced birational morphism $p: V' \rightarrow V$.	
\end{proposition}


To every Weil divisor 
\[ D = \sum\limits_{i = 1}^r a_i D_i, \quad a_i \in \mathbb{Z} \] 
on a toric variety $V$ to a complete fan $\Sigma$ is associated a polytope
\begin{align} \label{formula_polytope_assoc_to_divisor} 
P_D := \{ x \in M_{\mathbb{R}} \mid \langle x, \nu_i \rangle \geq -a_i, \quad \nu_i \in \Sigma[1] \},
\end{align}
which is a rational polytope, and which computes the global sections of $D$, that is (compare \cite[Prop.4.3.3]{CLS11})
\[ H^0(V, \mathcal{O}_{V}(D)) \cong  L(P_D). \]
Note that for $k \geq 1$ the polytope $P_{kD}$ associated to $k \cdot D$ equals $k \cdot P_D$.

\subsection{Nondegenerate hypersurfaces in toric varieties} \label{subsection_nondegenerate_hyp_tor_var}

Given a Laurent polynomial $f$ with Newton polytope $\Delta$, we denote the zero set in the torus $T$ by $Z_f$. For $F \subset M_{\mathbb{R}}$ a rational polytope we denote the closure of $Z_f$ in $\mathbb{P}_F$ by $Z_F$. 
\begin{definition}
	Given a Laurent polynomial $f$ with Newton polytope $\Delta$ we call $f$ nondegenerate with respect to $\Delta$ (or $\Delta$-regular) if $Z_f$ is smooth and $Z_{\Delta}$ intersects the toric strata of $\mathbb{P}_{\Delta}$ transversally. In this situation we sometimes also call $Z_{\Delta}$ nondegenerate.
\end{definition}

\begin{remark}
	\normalfont
By Bertini's Theorem the condition for $f \in L(\Delta)$ to be $\Delta$-regular is a Zariski open condition on the coefficients $(a_m)_{m \in M \cap \Delta}$.
\end{remark}

\begin{construction} \label{construction_lin_equiv_class_hyp}
	\normalfont
By (\cite[Prop.7.1]{Bat22}) for $f \in L(\Delta)$ the closure $Z$ in the toric variety $V$ to an $n$-dimensional complete fan $\Sigma$ is as Weil divisor linear equivalent to
\begin{align} \label{formula_lin_equiv_class_hyp}
	Z \sim_{lin} - \sum\limits_{\nu_i \in \Sigma[1]} \ord_{\Delta}(\nu_i) D_i.
\end{align} 
In particular the linear equivalence class of $Z$ is independent of $f$. This justifies our notation without mentioning $f$. \\
$Z$ is a Cartier divisor if and only if $\ord_{\Delta}$ is a \textit{support function}, that is
\[ \ord_{\Delta}: N_{\mathbb{R}} \rightarrow \mathbb{R} \]
is linear on each cone of $\Sigma$ and $\ord_{\Delta}(N) \subset \mathbb{Z}$. Similarly $Z$ is $\mathbb{Q}$-Cartier if we just have $\ord_{\Delta}(N) \subset \mathbb{Q}$.
\end{construction}

\section{Minimal models of toric hypersurfaces} \label{section_min_mod_of_tor_hyp}

\subsection{The Fine interior $F(\Delta)$}

In the article (\cite{Bat22}) it is shown how to construct minimal models of nondegenerate toric hypersurfaces. Since we will heavily exploit these constructions, we recall them here:

\begin{definition} \label{definition_Fine_interior}
Given an $n$-dimensional lattice polytope $\Delta \subset M_{\mathbb{R}}$ with presentation
\[ \{ x \in M_{\mathbb{R}} \mid \langle x, \nu_i \rangle \geq \ord_{\Delta}(\nu_i), \, \nu_i \in \Sigma_{\Delta}[1] \}, \]
we define the Fine interior $F(\Delta)$ of $\Delta$ as
\[ F(\Delta) := \{ x \in M_{\mathbb{R}} \mid \langle x, \nu \rangle \geq \ord_{\Delta}(\nu) + 1, \, \nu \in N \setminus \{0\} \}, \]	
\end{definition}

	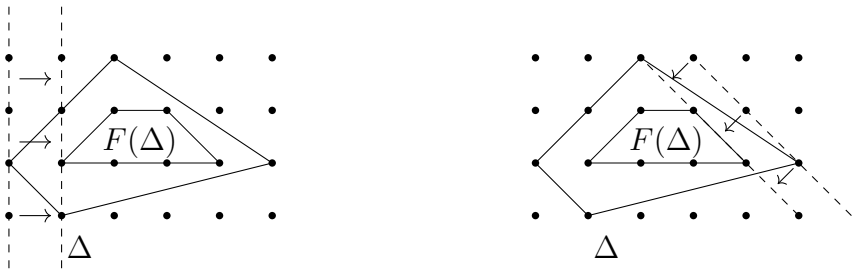
\begin{figure}[H]
	\begin{tikzpicture}[scale=0.7]
	
	\fill (-1,1) circle(2pt);
	\fill (-1,2) circle(2pt);
	\fill (-1,3) circle(2pt);
	\fill (-1,4) circle(2pt);
	\fill (0,1) circle (2pt);
	\fill (0,2) circle (2pt);	
	\fill (0,3) circle (2pt);
	\fill (0,4) circle (2pt);	
	\fill (1,1) circle (2pt);
	\fill (1,2) circle (2pt);	
	\fill (1,3) circle (2pt);
	\fill (1,4) circle (2pt);	
	\fill (2,1) circle (2pt);
	\fill (2,2) circle (2pt);	
	\fill (2,3) circle (2pt);
	\fill (2,4) circle (2pt);	
	\fill (3,1) circle (2pt);
	\fill (3,2) circle (2pt);	
	\fill (3,3) circle (2pt);
	\fill (3,4) circle (2pt);	
	\fill (4,1) circle (2pt);
	\fill (4,2) circle (2pt);	
	\fill (4,3) circle (2pt);
	\fill (4,4) circle (2pt);
	
	\draw (-1,2) -- (1,4);
	\draw (1,4) -- (4,2);
	\draw (-1,2)  -- (0,1);
	\draw (0,1) -- (4,2);
	
	\draw[very thin] (0,2)--(1,3);
	\draw[very thin] (1,3) -- (2,3);
	\draw[very thin] (2,3) -- (3,2);
	\draw[very thin] (3,2) -- (0,2);
	
	\draw[dashed] (-1,0) -- (-1,5);	
	\draw[arrows=->](-0.8,1)--(-0.2,1);
	\draw[arrows=->](-0.8,2.4)--(-0.2,2.4);
	\draw[arrows=->](-0.8,3.6)--(-0.2,3.6);
	\draw[dashed] (0,0) -- (0,5);

	\node [left] at (2.4,2.4) {{\bf $F(\Delta)$ }};
		\node [left] at (0.8,0.4) {{\bf $\Delta$}};
	
	\begin{scope}[xshift = 10cm]
	
	\fill (-1,1) circle(2pt);
	\fill (-1,2) circle(2pt);
	\fill (-1,3) circle(2pt);
	\fill (-1,4) circle(2pt);
	\fill (0,1) circle (2pt);
	\fill (0,2) circle (2pt);	
	\fill (0,3) circle (2pt);
	\fill (0,4) circle (2pt);	
	\fill (1,1) circle (2pt);
	\fill (1,2) circle (2pt);	
	\fill (1,3) circle (2pt);
	\fill (1,4) circle (2pt);	
	\fill (2,1) circle (2pt);
	\fill (2,2) circle (2pt);	
	\fill (2,3) circle (2pt);
	\fill (2,4) circle (2pt);	
	\fill (3,1) circle (2pt);
	\fill (3,2) circle (2pt);	
	\fill (3,3) circle (2pt);
	\fill (3,4) circle (2pt);	
	\fill (4,1) circle (2pt);
	\fill (4,2) circle (2pt);	
	\fill (4,3) circle (2pt);
	\fill (4,4) circle (2pt);
	
	\draw (-1,2) -- (1,4);
	\draw (1,4) -- (4,2);
	\draw (-1,2)  -- (0,1);
	\draw (0,1) -- (4,2);
	
	\draw[very thin] (0,2)--(1,3);
	\draw[very thin] (1,3) -- (2,3);
	\draw[very thin] (2,3) -- (3,2);
	\draw[very thin] (3,2) -- (0,2);
	
	\draw[dashed] (5,1) -- (2,4);	
	\draw[arrows=->](3.9,1.9)--(3.6,1.6);
	\draw[arrows=->](2.9,2.9)--(2.6,2.6);
	\draw[arrows=->](1.9,3.9)--(1.6,3.6);
	\draw[dashed] (4,1) -- (1,4);	
	
	\node [left] at (2.4,2.4) {{\bf $F(\Delta)$ }};
	\node [left] at (0.8,0.4) {{\bf $\Delta$}};

	\end{scope}
	\end{tikzpicture}
	\caption{Illustration of the construction of the Fine interior $F(\Delta)$ from $\Delta$.} \label{figure_examples_polytopes_Fine_interior}
\end{figure}

\begin{remark} \label{Remark_Jon_Fine_2_dim_Fine_interior}
	\normalfont
	The Fine interior was introduced by J. Fine in (\cite{Fine83}). In general it is only a rational polytope although if $\dim \, \Delta= 2$ it is always a lattice polytope, namely it equals the convex span of the interior lattice points of $\Delta$ (\cite[Prop.2.9]{Bat17}).
\end{remark}

\begin{definition}
	\normalfont
	The set of lattice points $\nu \in N \setminus \{0\}$ with
	\[ \ord_{F(\Delta)}(\nu) = \ord_{\Delta}(\nu) + 1 \]
	is called the support $S_F(\Delta)$ of $F(\Delta)$ to $\Delta$.
\end{definition}

\subsection{A diagram of toric morphisms}

There is another polytope $C(\Delta)$ which might be slightly larger than $\Delta$. $C(\Delta)$ has the nice property that all normal vectors to facets of $C(\Delta)$ belong to $S_F(\Delta)$.

\begin{definition}
	\normalfont
	The polytope
	\[ C(\Delta) := \{ x \in M_{\mathbb{R}} \mid \langle x,\nu \rangle \geq \ord_{\Delta}(\nu) \quad  \forall \, \nu \in S_{F}(\Delta) \} \]
	is called the canonical closure of $\Delta$. We call $\Delta$ canonically closed if $C(\Delta) = \Delta$. 
\end{definition}

\begin{lemma} \label{lemma_canonical_closure_characterization} (\cite[Cor.3.19, Prop.4.3]{Bat22} \\
	An $n$-dimensional lattice polytope $\Delta$ with $F(\Delta) \neq \emptyset$ is canonically closed if and only if $\Sigma_{\Delta}[1] \subset S_F(\Delta)$. $C(\Delta)$ is canonically closed.
\end{lemma}

\begin{remark}
	\normalfont
	The canonical closure $C(\Delta)$ is in general again just a rational polytope. In dimension $2$ we have $C(\Delta) = \Delta$ if $F(\Delta) \neq \emptyset$ (\cite[Prop.4.4]{Bat22}).
\end{remark}

\begin{definition} \label{construction_with_the_many_points} 	\normalfont
	Let $\Delta \subset M_{\mathbb{R}}$ be an $n$-dimensional lattice polytope with Fine interior $F(\Delta) \neq \emptyset$. Then we define $\tilde{\Delta}$ as Minkowski sum
	\[ \tilde{\Delta}:= C(\Delta) + F(\Delta). \]
\end{definition}

By (\cite[Thm.6.3]{Bat22}) we still have
\[ \Sigma_{\tilde{\Delta}}[1] \subset S_F(\Delta). \]
Thus we may define a simplicial fan $\Sigma$ with $\Sigma[1] = S_F(\Delta)$ and which refines $\Sigma_{\tilde{\Delta}}$. Let $\mathbb{P}$ be the toric variety to $\Sigma$ and $Y = Y_f$ the closure of $Z_f$ in $\mathbb{P}$. Then we arrive at diagrams of toric morphisms and their restrictions, where by abuse of notation we use the same letters for the maps and its restrictions:

\begin{equation}  \label{diagram_toric_morphisms}
\begin{tikzcd}
& \mathbb{P} \arrow{d}{\pi} \\
& \mathbb{P}_{\tilde{\Delta}} \arrow[swap]{dl}{\rho} \arrow{dr}{\theta} \\
\mathbb{P}_{C(\Delta)} && \mathbb{P}_{F(\Delta)}
\end{tikzcd}
\quad
\begin{tikzcd}
& Y \arrow{d}{\pi} \\
& Z_{\tilde{\Delta}} \arrow[swap]{dl}{\rho} \arrow{dr}{\theta} \\
Z_{C(\Delta)} && Z_{F(\Delta)}
\end{tikzcd}
\end{equation}

Here $\pi$ and $\rho$ are birational and $\theta$ is birational if $\dim \, F(\Delta) = \dim \, \Delta$.

\subsection{Minimal models}

\begin{theorem} (\cite[Cor.6.6]{Bat22}) \\ \label{theorem_min_models}
	Let $\Delta \subset M_{\mathbb{R}}$ be an $n$-dimensional lattice polytope with $F(\Delta) \neq \emptyset$. Then $\mathbb{P}$ has at most terminal singularities and the adjoint divisor $K_{\mathbb{P}} + Y$ is $\mathbb{Q}$-Cartier and nef.	
\end{theorem}

\begin{remark} \label{Proposition_normality} (\cite[Thm.7.5]{Bat22}) \\
	\normalfont
	Let $\Delta$ be an $n$-dimensional lattice polytope with $F(\Delta) \neq \emptyset$. Then $Y$ does not contain any $(n-2)$-dimensional torus orbit of $\mathbb{P}$ and by this is a normal variety. We may apply the adjunction formula
	\[ K_{Y} = (K_{\mathbb{P}} + Y)_{\vert{Y}}.  \]
\end{remark}

\begin{theorem} (\cite[Thm.8.2]{Bat22}) \\
	Let $\Delta$ be an $n$-dimensional lattice polytope with $F(\Delta) \neq \emptyset$. Then the closure $Y$ gets a minimal model of $Z_f$, i.e. $Y$ is a projective variety birational to $Z_f$, has at most terminal singularities and $K_{Y}$ is nef. 
\end{theorem}

See (\cite[Def.2.34]{KM98}) for a definition of terminal singularities. For a toric variety there is a combinatorial criterion for singularities to be terminal (see \cite[(1.11,1.12)]{Rei83}).

\begin{remark} (\cite[Cor.6.5, Thm.8.1]{Bat22})  \\ \normalfont
	In this article we just work with the minimal model $Y$ but note the following: In the same situation as in Theorem \ref{theorem_min_models} $\mathbb{P}_{\tilde{\Delta}}$ has at most canonical singularities and the morphism $\pi: \mathbb{P} \rightarrow \mathbb{P}_{\tilde{\Delta}}$ is crepant, that is $\pi^*(K_{\mathbb{P}_{\tilde{\Delta}}}) = K_{\mathbb{P}}$. Likewise $Z_{\tilde{\Delta}}$ has at most canonical singularities and $\pi:Y \rightarrow Z_{\tilde{\Delta}}$ is crepant.
\end{remark}

\begin{remark} \label{remark_polytope_Fine_int_assoc_to_divisor}
\normalfont
Assuming $F(\Delta) \neq \emptyset$ we have by Construction \ref{construction_lin_equiv_class_hyp} 
\[ Y \sim_{lin} - \sum\limits_{\nu_i \in S_F(\Delta)} \ord_{\Delta}(\nu_i) D_i, \quad K_{\mathbb{P}} = - \sum\limits_{\nu_i \in S_F(\Delta)} D_i.  \]
Thus to the divisor $Y + K_{\mathbb{P}}$ is associated the polytope
\[ \{ x \in M_{\mathbb{R}} \mid \langle x, \nu_i \rangle \geq \ord_{\Delta}(\nu_i) + 1, \quad \nu_i \in S_F(\Delta)  \},  \]
which is exactly $F(\Delta)$. In particular the lattice points of $F(\Delta)$ count the global sections of $Y + K_{\mathbb{P}}$.
\end{remark}

\section{The plurigenera of minimal models} \label{section_plurigenera}

First the Kodaira dimension of $Y$, which by definition is the Kodaira dimension of a resolution of singularities of $Y$, has already been computed: For $\Delta$ an $n$-dimensional lattice polytope with $F(\Delta) \neq \emptyset$ and $k := \dim \, F(\Delta)$ if $k=n$ then $\kappa(Y) = n-1$, else $\kappa(Y) = k$ (\cite[Thm.9.2]{Bat22}). In the following Theorem we also compute the plurigenera of $Y$:

\begin{theorem} \label{theorem_plurigenera}
Let $\Delta$ be an $n$-dimensional lattice polytope, where $n \geq 2$, with $k:= \dim \, F(\Delta) \geq 0$. Then for $m \geq 1$ the plurigenera $P_m(Y):= h^{0}(Y, m K_Y)$ are given by
	\[
P_m(Y) = \left\{ \begin{array}{lll} l(m \cdot F(\Delta)) - l^*((m-1) \cdot F(\Delta)), & k = n \\
l(m \cdot F(\Delta)) + l^*((m-1) \cdot F(\Delta)), & k = n-1 \\
l(m \cdot F(\Delta)) & k < n-1,

\end{array}\right.
\]
\end{theorem}

\begin{proof}
If we take cohomology groups of a divisor $D$ we always mean the cohomology groups of the sheaf $\mathcal{O}(D)$ associated to $D$. Since $\mathbb{P}$ is $\mathbb{Q}$-factorial every Weil divisor on $\mathbb{P}$ is $\mathbb{Q}$-Cartier. By Remark \ref{remark_polytope_Fine_int_assoc_to_divisor} we have 
\[ H^{0}(\mathbb{P}, m (K_{\mathbb{P}}+ Y)) \cong L(m \cdot F(\Delta)) \quad m \in \mathbb{N}_{>0}. \]
We use an ideal sheaf sequence for $Y \subset \mathbb{P}$ and apply the adjunction formula
\begin{align*}
& 0 \rightarrow H^{0}(\mathbb{P}, (m-1)(K_{\mathbb{P}}+Y) + K_{\mathbb{P}}) \rightarrow H^{0}(\mathbb{P}, m (K_{\mathbb{P}}+ Y)) \\
& \rightarrow  H^{0}(Y, mK_Y) \rightarrow H^{1}(\mathbb{P}, (m-1)(K_{\mathbb{P}}+Y) + K_{\mathbb{P}})	\rightarrow 0
\end{align*}
where
\[ H^1(\mathbb{P}, m(K_{\mathbb{P}} + Y)) = 0 \]
by Demazure's vanishing (\cite[Thm.9.2.3]{CLS11}) for the divisor $K_{\mathbb{P}} + Y$, which is nef by Theorem \ref{theorem_min_models}. By (\cite[Thm.9.2.10a)]{CLS11}) we  may apply Serre-duality to the divisor $K_{\mathbb{P}}$. Thus for $m=1$ we have $h^0(\mathbb{P}, K_{\mathbb{P}}) = 0$ and 
\[ h^1(\mathbb{P}, K_{\mathbb{P}}) = h^{n-1}(\mathbb{P}, \mathcal{O}_{\mathbb{P}}) = 0  \]
by Demazure's vanishing Theorem again. Thus $P_1(Y)$ is given by
\[ P_1(Y) = l(F(\Delta)) = l^*(\Delta). \]
For $m \geq 2$ we apply Serre duality to the divisor  $(m-1)(K_{\mathbb{P}} + Y)$: We have  splittings
\begin{align*}
	& H^{n-p}(\mathbb{P}, \mathcal{O}_{\mathbb{P}}(K_{\mathbb{P}} + Y)) \cong \bigoplus\limits_{m \in M} H^{n-p}(\mathbb{P}, \mathcal{O}_{\mathbb{P}}(K_{\mathbb{P}} + Y))_m \cdot \chi^m   \\
	& H^{p}(\mathbb{P}, \mathcal{O}_{\mathbb{P}}(-Y)) \cong \bigoplus\limits_{m \in M} H^{p}(\mathbb{P}, \mathcal{O}_{\mathbb{P}}(-Y))_m \cdot \chi^m 
\end{align*}

and by (\cite[Ex.9.12 formula (9.2.9)]{CLS11}) Serre duality restricts to a duality
\[ \big( H^{n-p}(\mathbb{P}, \mathcal{O}_{\mathbb{P}}(K_{\mathbb{P}} + Y))_m \big)^* \cong H^{p}(\mathbb{P}, \mathcal{O}_{\mathbb{P}}(-Y))_{-m}.  \]
Thus for $D:= K_{\mathbb{P}} + Y$
\begin{equation*}
	H^{n-p}(\mathbb{P}, \mathcal{O}_{\mathbb{P}}(D)) \cong \left\{ \begin{array}{ll} 0 & p \neq \dim \, P_D \\
		\bigoplus\limits_{m \in L^*(P_D)} \mathbb{C} \cdot \chi^m & p = \dim \, P_D
	\end{array}\right. , 	
\end{equation*}
by the Batyrev-Borisov vanishing Theorem (\cite[Thm.9.2.7]{CLS11}), where $P_D$ denotes the polytope associated to $D$. We get
\begin{align*}
H^{0}(\mathbb{P}, (m-1)(K_{\mathbb{P}}+Y) + & K_{\mathbb{P}})  \cong H^n(\mathbb{P}, -(m-1)(K_{\mathbb{P}} + Y))^* \\
& \cong \left \{\begin{array}{ll} 0, & \dim \, F(\Delta) \leq n-1 \\
L^{*}((m-1)F(\Delta)), & \dim \, F(\Delta) = n  \end{array}\right.
\end{align*}
and
\begin{align*}
H^{1}(\mathbb{P}, (m-1)(K_{\mathbb{P}}+Y) + & K_{\mathbb{P}}) \cong H^{n-1}(\mathbb{P}, -(m-1)(K_{\mathbb{P}} + Y))^* \\
& \cong \left \{\begin{array}{lll} L^{*}((m-1)F(\Delta)), & \dim \, F(\Delta) =n-1 \\
0, & \dim \, F(\Delta) \neq n-1  \end{array}\right.
\end{align*}
The result follows by adding the dimensions in the exact sequence.

\end{proof}

\begin{example} \normalfont
	If $\dim \, \Delta=2$ with $F(\Delta) \neq \emptyset$ then $Y:= Z_{\Delta}$ is a smooth curve. As noted in Remark \ref{Remark_Jon_Fine_2_dim_Fine_interior} $F(\Delta)$ equals the convex span of the interior lattice points of $\Delta$. There is the following distinction
	\[  
	Y = \left\{ \begin{array}{lll} \textrm{ elliptic curve} & \dim \, F(\Delta) = 0 \\
	\textrm{ hyperelliptic curve}, \, g(Y) \geq 2 & \dim \, F(\Delta) = 1 \\
	\textrm{ non-hyperelliptic curve}, \, g(Y) \geq 3 & \dim \, F(\Delta) = 2 
	\end{array}
	\right.
	\]
	
	If $\dim \, F(\Delta) = 1$ then then the morphism 
	\[ \theta \circ \pi: \mathbb{P} \rightarrow \mathbb{P}_{F(\Delta)} \cong \mathbb{P}^1 \]
	induces the double covering $Y \rightarrow \mathbb{P}^1$. Conversely if $Y:= Z_{\Delta}$ is hyperelliptic of genus $\geq 2$ then $K_Y$ defines the double covering and with (\ref{pullback_formula_adjoint_div}) it follows easily that $\dim \, F(\Delta) = 1$. The formula in Theorem \ref{theorem_plurigenera} gets either trivial or reduces to Pick's formula.
\end{example}

	\begin{remark}
		\normalfont
		For $Y$ a minimal surface of general type the plurigenera of $Y$ are given by
		\[ P_{m}(Y) = \chi(Y, \mathcal{O}_Y) + \frac{m(m-1)}{2} K_{Y}^2.  \]
		Similarly the plurigenera of surfaces $Y$ with $\kappa(Y) = 1$, that is properly elliptic surfaces, could be deduced from the canonical bundle formula (see \cite[Ch.1, Prop.3.22]{FrMo94}). But in general it is difficult to compute the plurigenera already of a $3$-fold, say with at most canonical singularities (compare (\cite[Cor.(10.3)]{Rei87})).
	\end{remark}
		
	\begin{remark}
	\normalfont
		We have to be careful with any conclusions from the formula in Theorem \ref{theorem_plurigenera}: If $p_g(Y) = 0$ it might well be the case that the plurigenera are not monotone increasing in $m$, since $F(\Delta)$ is just a rational polytope and in the formula we just count the lattice points (see \cite[Example 3.6]{Bat22} for an example).
					
	\end{remark}

\section{The volume of $K_Y$}

The lattice normalized volume $\Vol_{\mathbb{Z}}(F)$ of a rational polytope $F$ may be defined by
\begin{align} \label{formula_lattice_vol_asymptotic}
	\Vol_{\mathbb{Z}}(F) = \lim\limits_{m \rightarrow \infty} \frac{l(m \cdot F) \cdot (\dim \, F)!}{m^{\dim \, F}}.
\end{align}
This is a standard fact (\cite[Lemma 3.19]{BR09}) and normalized means that the standard $n$-simplex $\Delta_n$ has $\Vol_{\mathbb{Z}}(\Delta_n) = 1$. We refer to (\cite{BR09}) for details. The following formula follows easily from the formula for the plurigenera. It was first and independently found by V. Batyrev in (\cite[Thm.9.4]{Bat22}).

\begin{corollar} \label{corollary_volume_of_can_divisor}
	Let $\Delta$ be an $n$-dimensional lattice polytope with $F(\Delta) \neq \emptyset$. Then
	\[
	K_Y^{n-1} = \left\{ \begin{array}{lll} 
		\Vol_{\mathbb{Z}}(F(\Delta)) + \sum\limits_{Q \textrm{ facet of }F(\Delta)} \Vol_{\mathbb{Z}}(Q), & \dim \, F(\Delta) = n \\
		2 \cdot \Vol_{\mathbb{Z}}(F(\Delta)) & \dim \, F(\Delta) = n-1 \\
		0 & \dim \, F(\Delta) < n-1
		
	\end{array}\right .
	\]
\end{corollar}

\begin{proof}
	By (\cite[Remark after Def. 2.2.31]{Laz00}) we have
	\[ K_Y^{n-1} = \lim\limits_{m \rightarrow \infty} \frac{(n-1)! \cdot h^0(Y, mK_Y)}{m^{n-1}}.  \]
	\leavevmode
	\newline \newline
	It follows that $K_Y^{n-1} = 0$ if $\dim \, F(\Delta) < n-1$. If $\dim \, F(\Delta) = n-1$ then by (\cite[Thm.4.1]{BR09}) we have
	\[  \lim\limits_{m \rightarrow \infty} \frac{(n-1)! \cdot l^*((m-1)F(\Delta)}{m^{n-1}} =  \lim\limits_{m \rightarrow \infty} \frac{(n-1)! \cdot l(m \cdot F(\Delta))}{m^{n-1}} = \Vol_{\mathbb{Z}}(F(\Delta)) \]
	and the result follows. 
	\newline \newline
	If $\dim \, F(\Delta) = n$ then write
	\[ P_m(Y) = l(m F(\Delta)) - l((m-1)F(\Delta)) + l((m-1)F(\Delta)) - l^*((m-1) F(\Delta)).  \]
	By formula \ref{formula_lattice_vol_asymptotic} we have
	\[  \lim\limits_{m \rightarrow \infty} \frac{l(m F(\Delta)) - l((m-1)F(\Delta))}{m^{n-1}/(n-1)!} = \Vol_{\mathbb{Z}}(F(\Delta)).  \]
	Finally by (\cite[Thm.5.6]{BR09}) we have
	\[ \lim\limits_{m \rightarrow \infty} \frac{l((m-1) F(\Delta)) - l^*((m-1)F(\Delta))}{m^{n-1}/(n-1)!} = \sum\limits_{Q \textrm{ facet of }F(\Delta)} \Vol_{\mathbb{Z}}(Q).  \]
\end{proof}

\section{The pluricanonical mappings}

By Remark \ref{remark_polytope_Fine_int_assoc_to_divisor} for
\begin{align} \label{m_index_Fine_interior}
m := \min \{ n \in \mathbb{N}_{\geq 1} \mid n \cdot F(\Delta) \textrm{ is a lattice polytope} \},
\end{align} 
the divisor $m(K_{\mathbb{P}} + Y)$ is Cartier and nef and thus basepointfree. Thus so is the restriction $mK_Y$. Let $D_{F(\Delta)}$ denote the ample $\mathbb{Q}$-Cartier divisor on $\mathbb{P}_{F(\Delta)}$ associated to $F(\Delta)$. By (\cite[Thm.6.2.8]{CLS11}) and Remark \ref{remark_polytope_Fine_int_assoc_to_divisor} we have an equality
\begin{align} \label{pullback_formula_adjoint_div}
m(K_{\mathbb{P}} + Y) = (\theta \circ \pi)^*(m D_{F(\Delta)}). 
\end{align} 
Thus $m(K_{\mathbb{P}} + Y)$ induces the morphism $\theta \circ \pi: \mathbb{P} \rightarrow \mathbb{P}_{F(\Delta)}$. 

\begin{corollar} \label{theorem_can_model_surf_gen_type}
	Let $\Delta$ be an $n$-dimensional lattice polytope such that \\
	$\dim \, F(\Delta) = n$. Then $Z_{F(\Delta)}$ gets a canonical model of $Y$.
\end{corollar}

\begin{proof}
	The morphism $\theta \circ \pi: \mathbb{P} \rightarrow \mathbb{P}_{F(\Delta)}$ is birational and equal to the identity on the torus. Thus also $(\theta \circ \pi)_{\vert{Y}}$ defines a birational morphism $Y \rightarrow Z_{F(\Delta)}$. $Z_{F(\Delta)}$ is normal since else it would contain an $(n-2)$-dimensional torus orbit of $\mathbb{P}_{F(\Delta)}$. But then also $Y$ would contain an $(n-2)$-dimensional torus orbit of $\mathbb{P}$ contradiciting Remark \ref{Proposition_normality}. By the adjunction formula we get for any curve $C \subset Y$
	\[ (\theta \circ \pi)(C) = pt \Leftrightarrow K_Y.C = 0.  \]
	Thus $Z_{F(\Delta)}$ equals the canonical model of $Y$ (\cite[Thm.3-3-6]{Mat02}).
\end{proof}

At least if $C(\Delta) = \Delta$ there is a useful criterion for $k \cdot K_Y$ to be basepointfree for a given $k \geq 1$:

\begin{proposition} \label{propositon_bpf_plurican_maps}
	Let $\Delta$ be a $n$-dimensional lattice polytope with $F(\Delta) \neq \emptyset$ and $C(\Delta) = \Delta$. Assume that $k \geq 1$ and every edge of $k \cdot F(\Delta)$ contains a lattice point. Then $k \cdot K_Y$ is basepointfree.
\end{proposition}

\begin{proof}
	Denote for $w = (w_1,...,w_n) \in M$ by $\chi^w$ the character 
	\[ t= (t_1,...,t_n) \mapsto t_1^{w_1}...t_n^{w_n}  \]
	corresponding to $w$. Then
	\[ H^0(\mathbb{P}, k(K_{\mathbb{P}} + Y)) \cong \bigoplus\limits_{w \in M \cap k \cdot F(\Delta)} \chi^w \cdot \mathbb{C}.  \]
	Let $M \cap k \cdot F(\Delta) = \{ w_1,...,w_l \}$ and let $\phi_k$ denote the map associated to the linear system $\text{\textbar} k(K_{\mathbb{P}}+Y) \text{\textbar}$. Then on the torus $T = (\mathbb{C}^*)^n$ the map $\phi_k$ is given by 
	\[ t \mapsto (\chi^{w_1}(t),...,\chi^{w_l}(t)).  \]
	By (\cite[Prop.4.1.2]{CLS11}) the divisor of $\chi^{w}$ on $\mathbb{P}$ is \leavevmode \\
	\[ div(\chi^w) = \sum\limits_{\nu_i \in \Sigma[1]} \langle w, \nu_i \rangle D_i. \]
	Let $F$ be an $r$-dimensional face of $k \cdot F(\Delta)$, where $r \geq 1$, and let $\Gamma_1,...,\Gamma_l$ denote all facets of $k \cdot F(\Delta)$ containing $F$. Then by assumption we find an $w \in M \cap F$. Let $div_0(\chi^w)$ denote the zero divisor of $\chi^w$, then \leavevmode \\
	\[ div_0(\chi^w) = k(K_{\mathbb{P}} + Y) + div(\chi^w) = \sum\limits_{\nu_i \in \Sigma[1]} (-k \ord_{F(\Delta)}(\nu_i) + \langle w, \nu_i \rangle) D_i.  \]
	Say that $\nu_i$ is a normal vector to $\Gamma_i$, then we have
	\[ \langle w, \nu_i \rangle = k \ord_{F(\Delta)}(\nu_i), \quad i=1,...,l.  \]
	Thus $div_0(\chi^w)$ does not meet the torus orbit to $F$. Thus the only basepoints of $\phi_k$ might be torus fixed points. But since $C(\Delta) = \Delta$ the hypersurface does not pass through these torus fixed points. Thus $k(K_{\mathbb{P}} + Y)$ is basepointfree on $Y$ and so is $kK_Y$.
\end{proof}

If $\dim \, F(\Delta) = \dim \, \Delta$ then we have for $1 \leq k \leq m$ the following criterion for the morphism $\phi_k$ associated to $\text{\textbar} k(K_{\mathbb{P}} +Y) \text{\textbar}$ to be birational:

\begin{proposition}
	Let $\dim \, F(\Delta) = \dim \, \Delta = n$ and $m$ be as in (\ref{m_index_Fine_interior}). Given $1 \leq k \leq m$ assume that every lattice point in $m \cdot F(\Delta)$ may be written as $\mathbb{Z}$-linear combination of lattice points of $k \cdot F(\Delta)$. Then $\text{\textbar} k(K_{\mathbb{P}} + Y) \text{\textbar}$ and $\text{\textbar} kK_Y \text{\textbar}$ define birational maps.
\end{proposition}

\begin{proof}
	 The assumptions imply that there are $n+1$ affine linear independent points in $L(k F(\Delta))$. Thus the image of $\phi_k$ has dimension $n$. We may restrict global sections of $k(K_{\mathbb{P}} + Y)$ to $Y$ and restrict the maps to the torus and $Z_f$. But on the torus $\phi_k$ is a finite \textit{morphism}. Thus it is enough to show that $\phi_k$ is birational. Let $L(k F(\Delta)) = \{w_1,...,w_l\}$ and $L(m F(\Delta)) = \{w_1,...,w_s\}$ where $s \geq l$. Restricted to the torus $\phi_k$ factors
	\begin{align*}
	t \mapsto (\chi^{w_1}(t),...,\chi^{w_s}(t)) \rightarrow (\chi^{w_1}(t),...,\chi^{w_l}(t))
	\end{align*}
	with the second map a projection and the first birational onto its image. Under the assumption of the Proposition the projection map, restricted to the image of the first map, has degree one, since for $w \in L(m F(\Delta))$ we find $a_1,...,a_l \in \mathbb{Z}$ such that
	\[ \chi^{w}(t) = \prod_{i=1}^{l} (\chi^{w_i}(t))^{a_i}.  \]	
\end{proof}

\end{document}